\documentclass[10pt]{amsart}
\input{amssymb.sty}
\usepackage{hyperref} 

\title[Virtual $W^*$-superrigidity]{Examples of group actions which are virtually $W^*$-superrigid}
\author{Jesse Peterson}
\address{Department of Mathematics, Vanderbilt University, 1326 Stevenson Center, Nashville, TN 37240}
\email{jesse.d.peterson@vanderbilt.edu}
\date{}
\thanks{The author's research is partially supported by NSF Grant 0901510 and a grant from the Alfred P. Sloan Foundation}
\subjclass{}

\keywords{}
\dedicatory{}
\newtheorem{thm}{Theorem}[section]
\newtheorem{cor}[thm]{Corollary}
\newtheorem{lem}[thm]{Lemma}
\theoremstyle{definition}
\newtheorem{defn}[thm]{Definition}

\newtheorem{examp}[thm]{Example}

\newcommand{\G}{\Gamma}
\newcommand{\g}{\gamma}

\newcommand{\C}{{\mathbb C}}
\newcommand{\F}{{\mathbb F}}
\newcommand{\N}{{\mathbb N}}
\newcommand{\Z}{{\mathbb Z}}
\newcommand{\R}{{\mathbb R}}
\newcommand{\T}{{\mathbb T}}

\newcommand{\CC}{{\mathcal C}}
\newcommand{\HH}{{\mathcal H}}
\newcommand{\KK}{{\mathcal K}}

\newcommand{\NN}{{\mathcal N}}
\newcommand{\OO}{{\mathcal O}}

\newcommand{\UU}{{\mathcal U}}

\begin{document}
\begin{abstract}
We show that if $\G$ is a discrete group which does not have the Haagerup property but does have an unbounded cocycle into a $C_0$ 
representation and if $\G$ acts on a finite von Neumann algebra $B$ such that the inclusion $(B \subset B \rtimes \G)$ has the 
Haagerup property from below then any group-measure space Cartan subalgebra must have a corner which embeds into $B$ inside $B \rtimes \G$.

Taking the action to be trivial we produce examples of $II_1$ factors $N$ such that $N \overline \otimes M$ is not a group-measure space 
construction whenever $M$ is a finite factor with the Haagerup property.  Taking the action on a probability space with the Haagerup 
property from below we produce examples of von Neumann algebras which have unique group-measure space Cartan subalgebras.  Taking profinite actions of 
certain products of groups we use the unique Cartan decomposition theorem of N. Ozawa and S. Popa and the cocycle superrigidity theorem of A. 
Ioana to produce actions which are virtually $W^*$-superrigid.
\end{abstract}

\maketitle

\section{Introduction}

Given essentially free probability measure preserving actions of countable discrete groups $\G$ and $\Lambda$ on standard measure spaces $(X, \mu)$ and $(Y, \nu)$, a motivating problem in ergodic theory is to determine when the actions are equivalent in an appropriate sense.  The most natural notion of equivalence is a conjugacy between the actions, $\G \curvearrowright X$ and $\Lambda \curvearrowright Y$ are said to be conjugate if there exists a measure space isomorphism $\Delta:X \rightarrow Y$ and a group isomorphism
$\delta:\Gamma \rightarrow \Lambda$ such that $\Delta (\g x) = \delta(g) \Delta (x)$ for almost every $x \in X$.  

Associated to such an action one may construct the orbit equivalence relation $\mathcal R_{\G \curvearrowright X}$ on $X$ given by $x \sim y$
if and only if $x$ and $y$ have the same $\G$ orbit.  The actions $\G \curvearrowright X$ and $\Lambda \curvearrowright Y$ are said to be 
orbit equivalent (OE) if there is a measure space isomorphism $\Delta: X \rightarrow Y$ which preserves this equivalence relation, i.e.,
$\Delta(\G x) = \Lambda \delta(x)$ for almost every $x \in X$.  It is trivial to see that a conjugacy between the actions implements an orbit equivalence.

Also associated to such an action is the group-measure space construction $L^\infty(X,\mu) \rtimes \G$ of Murray and von Neumann \cite{murrayvonneumannII}.  $L^\infty(X, \mu) \rtimes \G$ is a finite von Neumann algebra which is a II$_1$ factor when the action is ergodic, 
moreover in this case $L^\infty(X, \mu)$ is a Cartan subalgebra of $L^\infty(X, \mu) \rtimes \G$, i.e., $L^\infty(X, \mu)$ is a maximal abelian
self adjoint subalgebra and the normalizer $\mathcal N_{L^\infty(X, \mu) \rtimes \G}(L^\infty(X, \mu)$ generates $L^\infty(X, \mu) \rtimes \G$
as a von Neumann algebra.  The actions $\G \curvearrowright X$ and $\Lambda \curvearrowright Y$ are said to be $W^*$ equivalent 
if the corresponding group-measure space constructions are isomorphic von Neumann algebras.

Singer showed in \cite{singer} that given a measure space isomorphism $\Delta: X \rightarrow Y$, $\Delta$ will implement an orbit equivalence
between the actions if and only if the corresponding von Neumann algebra isomorphism $\Delta_*:L^\infty(X, \mu) \rightarrow L^\infty(Y, \nu)$ 
given by $\Delta_*(f) = f \circ \Delta^{-1}$ extends to an isomorphism of the group-measure space constructions.  In particular, this gives the
following implications:
$$
{\rm Conjugacy} \ \implies \ {\rm Orbit \ Equivalence} \ \implies \ W^* \ {\rm Equivalence}.
$$

In general neither implication can be reversed, for example the theorem of Ornstein and Weiss \cite{ornsteinweiss} states that
any two free ergodic measure preserving actions of amenable groups are orbit equivalent, while the theorem of Connes and Jones \cite{connesjones}
shows that there are actions of groups which give rise to isomorphic group-measure space constructions but for which any isomorphism cannot carry one Cartan subalgebra onto the other and hence the actions are not orbit equivalent.  

A rigidity phenomenon in this setting is when there are 
general conditions which we may place on the groups and/or actions so that some of the implications above may be reversed.  The strongest form of rigidity is 
when all of the conditions are placed on only one side, say on $\G \curvearrowright X$, while the other group and action are allowed to be arbitrary.  In this setting we say
that the action $\G \curvearrowright X$ is superrigid.

Examples of OE-superrigidity (going from orbit equivalence to conjugacy) include those given by Furman \cite{furmanI, furmanII}, Popa \cite{popasuperrigid, popaspectralgap}, and Kida \cite{kidaI, kidaII}.  There is also a virtual OE-superrigidity result given by Ioana \cite{ioanaprofinite}.  Superrigidity going from $W^*$ equivalence to orbit equivalence has been given by Ozawa and Popa \cite{ozawapopaI, ozawapopaII}, unfortunately though there is
no intersection between the examples of Ozawa and Popa and and the OE-superrigidity examples above.  Hence finding examples of $W^*$-superrigidity (being able to 
reverse both implications) has remained elusive.

The purpose of this paper is to present further examples of superrigidity going from $W^*$ equivalence to orbit equivalence.
To state these examples we first introduce some notation.

\begin{defn}
Let $\CC$ be the class of groups $\G$ such that there exists an unbounded cocycle into a $C_0$ representation. 
\end{defn}

Examples of groups in $\CC$ include all groups with the Haagerup property, all non-trivial free products, and all groups $\G$ such that 
$\beta_1^{(2)}(\G) > 0$.  Examples of groups which are not in the class $\CC$ include all groups which do not have the Haagerup property and 
are a direct product of two infinite groups, and all groups which do not have the Haagerup property and contain an infinite normal abelian 
subgroup.

\begin{thm}[$W^*$ equivalence $\implies$ OE Superrigidity]\label{thm:introuniquegms}
If $\G \in \CC$ does not have the Haagerup property and acts freely and ergodically on a standard probability space $(X, \mu)$ such that $( 
L^\infty(X, \mu) \subset L^\infty(X, \mu) \rtimes \G )$ has the Haagerup property from below then $L^\infty(X, \mu)$ is the unique (up to unitary 
conjugacy) group-measure space Cartan subalgebra, i.e., if $\Lambda$ is a countable discrete group which acts freely on a probability space $(Y, 
\nu)$ and there is an isomorphism $\theta: L^\infty(Y, \nu) \rtimes \Lambda \rightarrow L^\infty(X, \mu) \rtimes \G = N$ then there exists a 
unitary $u \in \UU(N)$ such that $u \theta(L^\infty(Y, \nu)) u^* = L^\infty(X, \mu)$.
\end{thm}

The proof this theorem uses the connection between closable derivations and completely positive deformations established by Sauvageot in \cite{sauvageotder, sauvageotcp}, together with the infinitesimal analysis initiated in \cite{petersonl2}.  This should be considered as 
another entry in Popa's deformation/rigidity theory (see \cite{popaicm} for a survey).

By combining these techniques with the results of Ozawa-Popa and Ioana mentioned above we obtain a virtual $W^*$-superrigidity result.

\begin{thm}[Virtual $W^*$-superrigidity]\label{thm:introwesuper}
Let $\G_1 \in \CC$ be a finitely generated group which contains a subgroup $\G_0$ such that the pair $(\G_1, \G_0)$ has relative property (T).  
Let $\G_2$ be a finitely generated group which is weakly amenable with constant $1$, and which has a proper cocycle into a non-amenable 
representation.  Then any free ergodic profinite action of $\G = \G_1 \times \G_2$ which is ergodic when restricted to $\G_0$ and $\G_2$ is 
virtually $W^*$-superrigid.  That is to say if a group $\Lambda$ acts by free measure preserving transformations on a standard probability 
space $(Y, \nu)$ such that the von Neumann algebras $L^\infty(Y, \nu) \rtimes \Lambda$ and $L^\infty(X, \mu) \rtimes \G$ are isomorphic, then 
there exists an isomorphism $\theta: X \rightarrow Y$, two finite index subgroups $\G_f < \G$, $\Lambda_f < \Lambda$, a $\G_f$-ergodic 
component $X_f \subset X$ and a $\Lambda_f$-ergodic component $Y_f \subset Y$ such that $\theta(X_f) = Y_f$ and $\theta_{|X_f}$ is a conjugacy 
between the actions $\G_f \curvearrowright X_f$ and $\Lambda_f \curvearrowright Y_f$.
\end{thm}

Examples of a group actions which satisfies the hypotheses of Theorem \ref{thm:introwesuper} and are therefore virtually $W^*$-superrigid are 
the following.

\begin{examp}(Y. Shalom, \cite{shalom})\label{examp:shalom}
Let $n \geq 3$ be an odd integer, $p$ be a prime number and consider $\Z_p$ the ring of pro-$p$ integers.  Denote by $\lambda$ the Haar measure 
on the profinite (and in particular, compact) group $SL_n(\Z_p)$.  Consider $\Z$ as a dense subgroup of $\Z_p$ and also $SL_n(\Z)$ as a dense 
subgroup in $SL_n(\Z_p)$.

If $A_1, A_2, \ldots, A_k \in SL_n(\Z)$ and $w$ is a word in $\F_{k + 1}$ such that the closed set $C = \{ g \in SL_n(\Z_p) | w(g, A_1, \ldots, 
A_k) = e \}$ contains an open set, then since $SL_n(\Z)$ is dense in $SL_n(\Z_p)$ it follows that $C$ contains a set of the form $A\G$ where $A 
\in SL_n(\Z)$ and $\G < SL_n(\Z)$ is a finite index subgroup.  However for elements in $SL_n(\Z)$ the condition of being in the set $C$ is an
algebraic condition, and since $\G$ is Zariski dense in $SL_n(\C)$ it then follows that $w(B, A_1, \ldots, A_k) = A$, for all $B \in SL_n(\C)$.
  
By \cite{shalen} there exists elements in $SL_n(\C)$ which are free from $SL_n(\Z)$, therefore $C$ contains no open set and so 
by the Baire Category theorem it follows that 
the set of torsion free element in $SL_n(\Z_p)$ which are free from $SL_n(\Z)$ is dense.  In particular, by choosing such an element we have that 
$SL_n(\Z) * \Z$ acts on $SL_n(\Z_p)$ by left multiplication.

By Corollary 1.1 in \cite{breuillardgelander} We may densely embed a finite rank free group $F$ into $SL_n(\Z_p)$.  If we let $F$ act by 
inverse right multiplication then we obtain a measure preserving profinite action of $(SL_n(\Z) * \Z ) \times F$ on $(SL_n(\Z_p), \lambda)$ 
which is ergodic when restricted to $SL_n(\Z)$ or $F$ since the embeddings are dense.  Moreover since $n$ is odd $SL_n(\Z_p)$ has trivial 
center and hence this action is essentially free.
\end{examp}

After a preliminary version of this paper had been circulated, Popa and Vaes in \cite{popavaessuper} have been able
to produce a large class of explicit examples of actions which are $W^*$-superrigid.  Their results are also based on deformation/rigidity theory.

In the setting of group von Neumann algebras rather than group-measure space constructions we are also able to give new examples of 
factors which do not arise as a group-measure space construction.  The first example of such a factor is $L\mathbb F_n$ as established by
Voiculescu in \cite{voiculescucartan}.

\begin{thm}\label{thm:intronogms}
If $\G \in \CC$ is ICC and does not have the Haagerup property then $L\G \overline \otimes M$ is not a group-measure space construction 
whenever $M$ is a finite factor with the Haagerup property (including the case $M = \C$).
\end{thm}

The author would like to thank Adrian Ioana for suggesting Lemma \ref{lem:reassemble} which strengthens these results, and Yehuda Shalom for 
Example \ref{examp:shalom}.

\section{Preliminaries}\label{sec:prelim}
Throughout this paper all finite von Neumann algebras will be separable and come with a faithful normal trace which we will denote by $\tau$.

\subsection{Compact correspondences}

Here we review the notation of an $N$-$N$ correspondence to have compact coefficients.

\begin{defn}
Let $N$ be a finite von Neumann algebra and $\HH$ an $N$-$N$ Hilbert bimodule.  $\HH$ is said to be a {\bf compact correspondence} if it has compact 
matrix coefficients, i.e., given any bounded vector $\xi \in \HH$ the c.p. map $\Phi: N \rightarrow N$ which satisfies $\tau (\Phi(x) y) = 
\langle x \xi y, \xi \rangle$, for all $x,y \in N$ is compact when viewed as an operator on $L^2N$.
\end{defn}

\begin{defn}\label{defn:compactcorrespondence}
Let $N$ be a finite von Neumann $B \subset N$ a von Neumann subalgebra and $\HH$ an $N$-$N$ Hilbert bimodule.  $\HH$ is said to be 
{\bf compact relative to $B$} if given any sequence $x_n \in (N)_1$ such that $\| E_B(y x_n z) \|_2 \rightarrow 0$, for all $y,z \in N$ then $\langle x_n \xi 
y_n, \xi \rangle \rightarrow 0$, for any sequence $y_n \in (N)_1$.  Equivalently, for any bounded vector $\xi \in \HH$ we have that $\| 
\Phi_\xi (x_n) \|_2 \rightarrow 0$.
\end{defn}

We remark here that we do not require in Definition \ref{defn:compactcorrespondence} for the correspondence to have a non-zero $B$-central 
vector.  Hence if $B_1 \subset B_2 \subset N$ are von Neumann subalgebras and $\HH$ is a $N$-$N$ Hilbert bimodule which is compact relative to 
$B_2$ then $\HH$ is also compact relative to $B_1$.  Also note that if $B_1$ has finite index in $B_2$ then $\HH$ is compact relative to $B_1$ 
if and only if $\HH$ is compact relative to $B_2$. 

\begin{examp}
Let $N$ be a finite von Neumann algebra, then $L^2N \overline \otimes L^2N$ is a compact correspondence.  More generally if $B \subset N$ is a 
von Neumann subalgebra then $L^2\langle N, B \rangle$ is compact relative to $B$.
\end{examp}

\begin{examp}
Let $\G$ be a discrete group which acts on a finite von Neumann algebra $B$, let $N = B \rtimes \G$ be the crossed product construction and 
suppose that $\pi: \G \rightarrow \OO(\KK)$ is a $C_0$ representation.  Then the associated $N$-$N$ bimodule given by $\HH_\pi = \KK \overline 
\otimes L^2N$ with actions satisfying 
$$
b_1u_{\gamma_1} (\xi \otimes \eta) b_2u_{\gamma_2} = (\pi(\gamma) \xi ) \otimes (b_1 u_{\gamma_1} \eta
b_2 u_{\gamma_2} ),
$$
for each $\xi \in \KK, \eta \in L^2N, \gamma_1, \gamma_2 \in \G, b_1, b_2 \in B$,
is compact relative to $B$.
\end{examp}

In general, if $\HH$ is a Hilbert $N$-$N$ bimodule which is compact relative to $B$ and $\Phi: N \rightarrow N$ is a unital subtracial c.p. map 
with associated pointed bimodule $(\HH_\Phi, \xi_\Phi)$ then the Hilbert $N$-$N$ bimodule generated by vectors of the form $\xi_\Phi \otimes_N 
\xi \otimes_N \xi_\Phi \in \HH_{\Phi^*} \otimes_N \HH \otimes \HH_\Phi$ will no longer be compact relative to $B$.  Consider the case when 
$\Phi$ is an automorphism for example.  However, if $\Phi$ is $B$-bimodular then it is easy to see that this does indeed hold.

\subsection{Closable derivations}

Given a finite von Neumann algebra $N$ and a Hilbert $N$-$N$ bimodule $\HH$, a closable derivation on $N$ into $\HH$ is a linear a map 
$\delta:D(\delta) \rightarrow \HH$ such that $D(\delta)$ is a weakly dense $*$-subalgebra of $N$ and for every $x, y \in D(\delta)$ the product 
rule holds:
$$
\delta(xy) = x\delta(y) + \delta(x)y.
$$
 We 
say that $\delta$ is real if there is a conjugate linear isometric involution $J$ on $\HH$ such that $J( x \delta(y) z) = z^* \delta(y^*) x^*$, 
for all $x,y,z \in D(\delta)$.

When given a closable real derivation $\delta: D(\delta) \rightarrow \HH$ we will use the same notation as in \cite{petersonl2} and associate to 
$\delta$ two deformations of unital $\tau$-symmetric completely positive maps
$$
\eta_\alpha =  \frac{\alpha}{\alpha + \delta^*\overline \delta}, \ \ \ \zeta_\alpha = \eta_\alpha^{1/2}.
$$  
The fact that these maps are completely positive follows from \cite{sauvageotcp}.

In \cite{petersonl2} the contraction $\tilde \delta_\alpha: N \rightarrow \HH$ was given by $\tilde \delta_\alpha (x) = \alpha^{-1/2} \overline\delta 
\circ \zeta_\alpha (x)$,  here however we will find it more useful instead to use the contraction $\delta_\alpha: N \rightarrow \HH_\alpha 
\otimes_N \HH \otimes_N \HH_\alpha$ given by 
$$
\delta_\alpha(x) = \alpha^{-1/2} \xi_\alpha \otimes_N \overline\delta \circ \zeta_\alpha(x) \otimes_N \xi_\alpha,
$$
where $(\HH_\alpha, \xi_\alpha)$ is the pointed correspondence associated to the completely positive map $\zeta_\alpha$ and $\otimes_N$ denotes 
Connes' tensor product of $N$-$N$ bimodules. 
Note that $\delta_\alpha(x)$ is well defined since $\xi_\alpha$ is a tracial vector, (see Section 1.3.1 in \cite{popacorrespondence}).

The advantage of using $\delta_\alpha$ is that the approximation property (Lemma 3.3 in \cite{petersonl2} or 
Lemma 4.1 in \cite{ozawapopaII}) has a nicer form than it does for $\tilde \delta_\alpha$.  Specifically, from Lemma 4.1 in \cite{ozawapopaII} it follows that for 
every $a, x \in N$ we have
$$
\| a \delta_\alpha(x) - \delta_\alpha(ax) \| \leq 50 \| x \|_\infty \| a \|_\infty^{1/2} \| \delta_\alpha (a) \|^{1/2}
$$
and 
$$
\| \delta_\alpha(xa) - \delta_\alpha(x)a \| \leq 50 \| x \|_\infty \| a \|_\infty^{1/2} \| \delta_\alpha (a) \|^{1/2}.
$$
Moreover, the correspondences $\HH_\alpha$ inherit many properties of $\HH$ useful in applications, e.g. compactness, non-amenability, being 
weakly contained in the coarse correspondence, etc.

An important class of closable real derivations on a crossed product von Neumann algebra are the ones which come from group cocycles.  If $\G$ 
is a countable discrete group which acts by $\tau$-preserving automorphisms on a finite von Neumann algebra $B$, $\pi: \G \rightarrow \OO(\KK)$ 
is an orthogonal representation, and $b:\G \rightarrow \KK$ is a cocycle ($b(\gamma_1 \gamma_2) = b(\gamma_1) + \pi(\gamma_2)b(\gamma_1)$ for 
all $\gamma_1, \gamma_2 \in \G$) then the derivation $\delta_b$ from $B \G \subset B \rtimes \G$ into the $(B \rtimes \G)$-$(B \rtimes \G)$ 
bimodule $\HH_\pi$ which satisfies $\delta_b(xu_\gamma) = b(\gamma) \otimes xu_\gamma$ for all $\gamma \in \G$, $x \in B$ is a closable real 
derivation.  From the remarks above we see that in this case the contractions ${\delta_b}_\alpha$ map into a $L\G$-$L\G$ Hilbert bimodules 
which are compact relative to $B$.

If $A \subset N$ is a von Neumann subalgebra, then Corollary 2.3 in \cite{popamalleableII} states that no corner of $A$ embeds into $B$ inside 
$N$ precisely when there exists a sequence of unitaries $u_n \in \UU(A)$ such that $\| E_B( x u_n y ) \|_2 \rightarrow 0$, for all $x,y \in N$.  
This fact, together with Theorem 4.5 in \cite{petersonl2} yields the following result.

\begin{thm}\label{thm:boundednormalizer}
Let $N$ be a finite von Neumann algebra, $B,A \subset N$ von Neumann subalgebras, $\HH$ a Hilbert $N$-$N$ bimodule which is compact relative to 
$B$ and $\delta:N \rightarrow \HH$ a closable real derivation.  If no corner of $A$ embeds into $B$ inside $N$, and the associated deformation 
$\eta_\alpha$ converges uniformly in $\| \cdot \|_2$ to the identity on $(A)_1$ then $\eta_\alpha$ converges uniformly in $\| \cdot \|_2$ to 
the identity on $(\NN_N(A)'')_1$.
\end{thm}

\subsection{The Haagerup property}

Recall that a countable discrete group $\G$ is said to have the Haagerup property or be a-T-menable in the sense of Gromov if any of the 
following equivalent conditions hold (see \cite{CCJJV} and \cite{propertyT} for more on the Haagerup property and other equivalent 
formulations):

\vskip .1in
\noindent
$(a)$.  There exists a $C_0$ representation of $\G$ which almost contains invariant vectors.
 \vskip .05in
\noindent
$(b)$.  There exists a sequence of $C_0$ positive definite functions on $\G$ which pointwise converge to $1$.
\vskip .05in
\noindent
$(c)$.  There exists a proper conditionally negative definite function on $\G$.
\vskip .05in
\noindent
$(d)$.  There exists a proper cocycle on $\G$.
\vskip .05in
\noindent
$(e)$.  There exists a proper cocycle on $\G$ into a $C_0$ representation.

\vskip .1in

If one does not have a sequence of $C_0$ positive definite functions as in condition $(b)$ above but one does have positive definite functions 
which are sufficiently far away from being uniformly close to $1$ on any infinite set, then by modifying an argument in \cite{akemannwalter} it 
is still possible to conclude that there is a proper conditionally negative definite function on $\G$.

\begin{lem}\label{thm:notuniformhaagerup}
Suppose $\G$ is a countable discrete group and $\varphi_n: \G \rightarrow \C$ is a sequence of positive definite functions such that $\lim_{n 
\rightarrow \infty} | 1 - \varphi_n (\gamma) | = 0$, for all $\gamma \in \G$, and for any sequence of numbers $m_k \in \N$ there exists a 
sequence $n_k \in \N$ such that $n_k \geq m_k$, for all $k \in \N$ and $\varphi_{n_k}$ does not converge uniformly on any infinite subset of 
$\G$, then $\G$ has the Haagerup property.
\end{lem}
\begin{proof}
Note that by considering the positive definite functions $\gamma \mapsto | \Re ( \varphi_n ( \gamma) / \varphi_n(e) ) |^2$ we may assume that 
$\varphi_n$ takes real values between $0$ and $1$.  Let $\{ \gamma_n \}_n$ be an enumeration of $\G$, since $\lim_{n \rightarrow \infty} | 1 - 
\varphi_n (\gamma) | = 0$, for all $\gamma \in \G$, we may construct a sequence $m_k$ such that for all $n \geq m_k$ we have $1 - 
\varphi_n(\gamma_j) < 4^{-k}$, whenever $1 \leq j \leq k$.

By assumption we may then find a sequence $n_k$ such that $n_k \geq m_k$ and $\varphi_{n_k}$ does not converge uniformly on any infinite subset 
of $\G$.  Define the conditionally negative definite function $\psi$ by $\psi(\gamma) = \Sigma_{k = 1}^\infty 2^k( 1  - \varphi_{n_k}(\gamma) 
)$.  Note that since $n_k \geq m_k$, we have $\psi(\gamma) < \infty$, for all $\gamma \in \G$, and if $\psi$ is bounded by $K$ on a set $X 
\subset \G$ then for all $x \in X$ we have $1 - \varphi_{n_k}(x) \leq K2^{-k}$, for all $k \in \N$ hence $X$ must be finite and so $\psi$ is 
proper.
\end{proof}

\section{The Haagerup property from below for inclusions of groups and von Neumann algebras}

\begin{defn}\label{defn:gprelhaagerup}
Let $\G$ be a countable discrete group and $\Lambda < \G$ a subgroup.  The pair $(\G, \Lambda)$ has the {\bf Haagerup property from below} if there 
exists a sequence of positive definite functions $\psi_n: \G \rightarrow \C$ such that $\lim_{n \rightarrow \infty} | 1 - \psi_n(\gamma) | = 
0$, for all $\gamma \in \G$, and ${\psi_n}_{|\Lambda} \in C_0(\Lambda)$, for all $n \in \N$.
\end{defn}

By using standard arguments on the relationship between cocycles and semigroups of positive definite functions it follows that an equivalent characterization 
of the Haagerup property from below is the existence of a cocycle on $\G$ which is proper when restricted to $\Lambda$, (see for example Theorem 2.1.1 in \cite{CCJJV}).  

Note that if $\Lambda$ is a normal subgroup such that $\G/\Lambda$ has the Haagerup property then $(\G, \Lambda)$ has the Haagerup property from below 
if and only if $\G$ has the Haagerup property.  Also note that by standard methods one can show that the definition above is 
equivalent to saying that $\G$ has a cocycle in some representation which is proper when restricted to $\Lambda$.

It is shown in \cite{cornulierstaldervalette} that the wreath product of $\Z/2\Z$ by $\F_2$ gives rise to a group which has the Haagerup 
property and thus $((\Z/2\Z) \wr \F_2, \oplus_{\F_2} (\Z/2\Z))$ has the Haagerup property from below.  More generally, in 
\cite{cornulierstaldervaletteII} it was shown that if $\G$ and $\Lambda$ both have the Haagerup property then $\Lambda \wr \G$ also has the 
Haagerup property.  If one starts with a group $\G$ which does not have the Haagerup property and $\Lambda$ is not the trivial group then the 
pair $(\Lambda \wr \G, \oplus_{\G} \Lambda )$ will never have the Haagerup property from below.  Indeed in this case given any cocycle on 
$\Lambda \wr \G$ there must be an infinite subset $X$ of $\G$ on which the cocycle is bounded and hence given $a \in \Lambda$ $a \not= e$ the 
cocycle will be bounded when restricted to the infinite set $\{ a \delta_g \}_{g \in X} \subset \G^{\Lambda} = \oplus_\G \Lambda$.

If we consider generalized wreath products then it is possible to still have the Haagerup property from below while still preserving the weak 
mixingness of the action.  We can achieve this roughly by requiring that our action has finite orbits when restricted to a subset which 
obstructs the Haagerup property.  

We say that a group $\G$ is residually Haagerup if for every $\g \in \G$, $\g \not= e$ there exists a 
homomorphism $\pi: \G \rightarrow H$ into a group $H$ with the Haagerup property such that $\pi(\g) \not= e$.

\begin{thm}\label{thm:bernoullihaagerup}
Let $\G$ be a residually Haagerup group, and $\Lambda$ be a group with the Haagerup property.  Consider a decreasing family in $\G$ of normal 
subgroups $H_n$, $n \geq 1$ such that $\G / H_n$ has the Haagerup property for all $n \in \N$.  $\G$ acts on the quotient space $\chi_n = \G / 
H_n$ by left multiplication and hence $\G$ also acts on the disjoint union $\chi = \coprod_n \G / H_n$ and we may consider the generalized 
wreath product $\Lambda \wr_\chi \G$.  If $A = \oplus_\chi \Lambda \lhd \Lambda \wr_\chi \G$ then the pair $(\Lambda \wr_\chi \G, A)$ has the 
Haagerup property from below.  Moreover the corresponding action of $\G$ on $\ell^2A$ is free if $\cap_n H = \{ e \}$ and is weakly mixing if 
$[\G : H_1] = \infty$.
\end{thm}

\begin{proof}
If we denote by $A_n = \oplus_{x \in \chi_n} \Lambda$ then $A_n$ is preserved under the action of $\G$ and $A_n \rtimes \G = \Lambda \wr_{\chi_n} 
\G$.
Moreover $A = \oplus_n A_n$ and hence it is enough to show that $\oplus_{n = 1}^N A_n$ has the Haagerup property from below in $(\oplus_{n = 1}^N 
A_n ) \rtimes \G$, for all $N \in \N$.

If we fix $N \in \N$ then $(\oplus_{n = 1}^N A_n ) \rtimes (\G/H_N)$ contains $\oplus_{n = 1}^N A_n $ and has the Haagerup property by  
\cite{cornulierstaldervaletteII}.  Thus if we consider a sequence of $C_0$ positive definite functions on $(\oplus_{n = 1}^N A_n ) \rtimes \G$ 
which converges pointwise to $1$ then the pullback of this sequence to $(\oplus_{n = 1}^N A_n ) \rtimes \G$ shows that the pair $( (\oplus_{n = 
1}^N A_n ) \rtimes \G, \oplus_{n = 1}^N A_n )$ has the Haagerup property from below.

The conditions for freeness and weak mixingness are simple to verify, (see Lemma 4.5 in \cite{popasuperrigid}). 
\end{proof}

We will define the Haagerup property from below for von Neumann algebras only in the crossed product setting.

\begin{defn}\label{defn:relhaagerup}
Let $\G$ be a countable discrete group and $\sigma: \G \rightarrow {\rm Aut}(B)$ and trace preserving action on a finite von Neumann algebra 
$B$.  Let $N = B \rtimes \G$, we say that the inclusion $(B \subset N)$ has the {\bf Haagerup property from below} if there exists a sequence $\Psi_n:N 
\rightarrow N$ of unital completely positive maps such that $\lim_{n \rightarrow \infty} \| \Psi_n(x) - x \|_2 = 0$, for all $x \in N$, 
$\Psi_n$ leaves the subspace $L^2Bu_\gamma$ invariant for all $\gamma \in \G$, and $\xi \mapsto \Psi_n (E_B (\xi u_\gamma^*))$ defines a 
compact operator for all $\gamma \in \G$.
\end{defn}

\begin{examp}
If the action of $\G$ on $B$ is profinite, i.e., there exists a sequence of finite dimensional $\G$-invariant subalgebras $B_n \subset B$ such 
that $\cup_n B_n$ is dense in $B$ then $(B \subset N)$ has Haagerup property from below.  Indeed, in this case we may set $\Psi_n = E_{B_n 
\rtimes \G}$ which satisfy the conditions above.  Also note that if $\G$ acts trivially on $B$ and $N = B \overline \otimes L\G$, then $(B 
\subset \G)$ has the Haagerup property from below if and only if $B$ has the Haagerup property.  
\end{examp}

If $\G$ has the Haagerup property and $\G$ acts on a finite von Neumann algebra $B$ such that 
$(B \subset B \rtimes \G)$ has the Haagerup property from below, then it is easy to check that $B \rtimes \G$ 
itself has the Haagerup property.  We will see below in 
Corollary \ref{cor:vnhaagerupprop} that the converse also holds in the case when $B$ is abelian.

The following fact follows from the definitions above using standard results regarding representations and correspondences. 

\begin{thm}\label{thm:gpvnrelhaagerup}
Let $\G$ be a countable discrete group which acts on a countable discrete group $\Lambda$.  Then $(\Lambda \rtimes \G, \Lambda)$ has the 
Haagerup property from below if and only if the inclusion $(L\Lambda \subset L\Lambda \rtimes \G)$ has the Haagerup property from below.
\end{thm}

\begin{proof}
Suppose $(\Lambda \rtimes \G, \Lambda)$ has the Haagerup property from below, then there exists a sequence of real valued positive definite 
functions $\varphi_n: \Lambda \rtimes \G \rightarrow \R$ such that $\varphi_n(e) = 1$ and $\varphi_n$ is $C_0$ when restricted to $\Lambda$.

If we consider the sequence of c.p. maps which satisfy $\phi_n(u_x) = \varphi_n(x) u_x$, for all $x \in \Lambda \rtimes \G$ then one can check 
that $\phi_n$ is a sequence which satisfies the conditions in Definition \ref{defn:relhaagerup}.

Conversely, If the inclusion $(L\Lambda \subset L\Lambda \rtimes \G)$ has the Haagerup property from below and $\phi_n: N \rightarrow N$ is a 
sequence of c.p. maps which satisfy Definition \ref{defn:relhaagerup}, then the positive definite functions given by $\varphi_n(x) = 
\tau(\phi_n(u_x)u_x^*)$ satisfy the conditions of Definition \ref{defn:gprelhaagerup}.
\end{proof}

A semigroup $\{ \Phi_t \}_t$ of completely positive maps $N$ has the useful property that $\tau( \Phi_t (x) x^* )$ increases as $t$ decreases 
for all $x \in N$.  We will want to exploit this in Theorem \ref{thm:nocornerhaagerup} below and so we will want that in Definition 
\ref{defn:relhaagerup} above one can take the sequence to be given by a semigroup.  Note that in the case when $N$ come from a subgroup/group 
inclusion as in Theorem \ref{thm:gpvnrelhaagerup}, 
then we may consider the c.p. deformations to be coming from the group.  In particular, this means 
that the c.p. maps will commute and the result becomes relatively simple.  In the general case one can follow the arguments in 
\cite{sauvageotfeller} and \cite{jolissaintmartin} where this is shown for the usual Haagerup property.

\begin{lem}\label{lem:haagerupsemigroup}
Let $\G$ be a countable discrete group and $\sigma: \G \rightarrow {\rm Aut}(B)$ and trace preserving action on a finite von Neumann algebra 
$B$.  Let $N = B \rtimes \G$, then the inclusion $(B \subset N)$ has the Haagerup property from below if and only if there exists a pointwise 
strongly continuous semigroup $\Phi_t: N \rightarrow N$ of unital symmetric completely positive maps such that $\Phi_t$ leaves the subspace 
$L^2Bu_\gamma$ invariant for all $\gamma \in \G$ and $\xi \mapsto \Phi_t( E_B (\xi u_\gamma^*))$ defines a compact operators for all $\gamma 
\in \G$.
\end{lem}

\section{Deformation/rigidity arguments}

Recall that the class $\CC$ was defined to be the class of all countable discrete groups which admit an unbounded cocycle into a $C_0$ representation.

Suppose $\G$ is a countable discrete group in the class $\CC$, and $\sigma: \G \rightarrow {\rm Aut}(B)$ is a trace preserving action on a 
(possibly atomic) finite von Neumann algebra $B$.  Suppose that $N = B \rtimes \G$ is a $II_1$ factor and that the inclusion $(B \subset N)$ 
has the Haagerup property from below.  Let $\Psi_k: N \rightarrow N$ be a  sequence of unital completely positive maps such that $\lim_{k 
\rightarrow \infty} \| \Psi_k(x) - x \|_2 =0$, for all $x \in N$, $\Psi_k$ leaves the subspace $L^2Bu_\gamma$ invariant for all $\gamma \in 
\G$, and $\xi \mapsto \Psi_k( E_B( \xi u_\gamma^*))$ defines a compact operator for all $k \in \N, \gamma \in \G$. 

Let $b:\G \rightarrow \KK$ be an unbounded cocycle into a $C_0$-orthogonal representation and let $\delta: B\G \rightarrow \KK \overline 
\otimes L^2N$ be the associated real closable derivation, also let $\eta_\alpha = \alpha / (\alpha + \delta^* \overline \delta)$ and 
$\zeta_\alpha = {\eta_\alpha}^{1/2}$ be the associated completely positive deformations.

\begin{thm}\label{thm:tech}
Using the hypotheses and notation above if $A \subset N$ is a Cartan subalgebra and $\{ u_n \}_n \subset \UU(N)$ is a sequence of unitaries 
such that the following conditions are satisfied:
\vskip .1in
\noindent
$(a)$.  $\lim_{k \rightarrow \infty} \sup_n \| u_n - \Psi_k ( u_n ) \|_2 = 0$.
\vskip .1in
\noindent
$(b)$.  $\lim_{\alpha \rightarrow \infty} \sup_n \| u_n - \eta_\alpha (u_n) \|_2 = 0$.
\vskip .1in
\noindent
$(c)$.  $\{ u_n \}_n$ is not $\| \cdot \|_2$-precompact.
\vskip .1in
\noindent
$(d)$.   $\{ u_n a u_n^* \}_n$ is $\| \cdot \|_2$-precompact, for all $a \in \UU(A)$.
\vskip .1in
Then a corner of $A$ embeds into $B$ inside $N$.
\end{thm}

\begin{proof}
We will assume that no corner of $A$ embeds into $B$ inside $N$ and then under the hypotheses above show that the deformation $\eta_\alpha$ 
must converge uniformly to the identity on $(N)_1$ and hence the cocycle $b$ must have been bounded. 

Since $A \not\preceq_N B$ it follows from Theorem \ref{thm:boundednormalizer} that if $\eta_\alpha$ converges uniformly on $\UU(A)$ then  
$\eta_\alpha$ also converges uniformly on $(N)_1$, thus we just need to show that $\eta_\alpha$ converges uniformly on $\UU(A)$.

By taking a subsequence if necessary we may assume that $\{ u_n \}_n$ is not $\| \cdot \|_2$-precompact and that $u_n$ converges weakly to an 
element $x \in (N)_1$.  Note that since $u_n$ does not converge strongly to $x$ we have that $x$ is not a unitary element and hence there is a 
non-zero spectral projection $p$ of $|x|$ such that $\| xp \|_\infty < 1$.

By condition $(a)$ we have that $\{ E_B(u_n u_\gamma^* ) \}_n$ is $\| \cdot \|_2$-precompact for all $\gamma \in \G$ and hence again by taking 
a subsequence we will assume that $E_B(u_n u_\gamma^*)$ converges in $\| \cdot \|_2$, for all $\gamma \in \G$.  Hence if $F \subset \G$ is a 
finite set and $P_F$ denotes the orthogonal projection onto $\overline {\rm sp} \{ Bu_\gamma | \gamma \in F \}$ then $\| P_F(u_n) - P_F(x) \|_2 
\rightarrow 0$, for all $F \subset \G$ finite.  Indeed by orthogonality it is enough to check the case when $F = \{ \g \}$ has a single element,
in this case $E_B(u_n u_\g^*) u_\g$ is Cauchy in $\| \cdot \|_2$ and converges weakly to $E_B(x u_\g^*)u_\g$, hence we must have that
$E_B(u_n u_\g^*) u_\g$ converges to $E_B(x u_\g^*)u_\g$ in $\| \cdot \|_2$.

By Lemma 4.1 in \cite{ozawapopaII} we then have that for all $a \in \UU(A)$
\begin{equation}
\limsup_{n \rightarrow \infty} | \ \| \delta_\alpha(pa) \|^2 - \langle u_n \delta_\alpha (pa) u_n^*, \delta_\alpha( u_n a u_n^*) \rangle \ | 
\leq 100 \sup_n \| \delta_\alpha (u_n) \|^{1/2} + 100 \| \delta_\alpha (p) \|^{1/2}.
\end{equation}

Also, since $\{ u_n a u_n^* \}_n$ is $\| \cdot \|_2$-precompact and $\delta_\alpha$ is a bounded operator,
we have that $\{ \delta_\alpha( u_n a u_n^* ) \}_n$ is also precompact.  Therefore, since as we noted in Section \ref{sec:prelim},
the $N-N$ correspondence $\HH_\alpha$ is compact relative to $A$ we have that 
$$
\lim_{n \rightarrow \infty} \langle (u_n - x) \delta_\alpha (pa) u_n^*, \delta_\alpha( u_n a u_n^*) \rangle = 0.
$$

Combining these facts, if we let $\varepsilon_\alpha = 100 \sup_n \| \delta_\alpha(u_n) \|^{1/2} + 100 \| \delta_\alpha (p) \|_2^{1/2}$
then we have
$$
\| \delta_\alpha (pa) \|_2^2 \leq \limsup_n | \langle u_{n} \delta_\alpha (pa) u_{n}^*, \delta_\alpha( u_{n} a u_{n}^*) \rangle
|
+ \varepsilon_\alpha
$$
$$
= \limsup_n | \langle x \delta_\alpha (pa) u_{n}^*, \delta_\alpha( u_{n} a u_{n}^*) \rangle |
+ \varepsilon_\alpha 
$$
$$
\leq \limsup_n | \langle xp \delta_\alpha (pa), u_{n}\delta_\alpha( a ) \rangle |
+ 2\varepsilon_\alpha 
$$
$$
\leq \|x p \|^2 \| \delta_\alpha (pa) \|_2^2 + 3\varepsilon_\alpha,
$$
where the last inequality follows from weak continuity of the left action of $N$ on $\HH$, the fact that $[ |x|, p ] = 0$, and again applying Lemma 4.1 in \cite{ozawapopaII}.

Hence we have shown that 
$$
\| \delta_\alpha(pa) \|_2 \leq (1 - \| x p \|^2 )^{-1} 3\varepsilon_\alpha ,
$$
for each $a \in \UU(A)$, where $\varepsilon_\alpha \rightarrow 0$ as $\alpha \rightarrow \infty$.

If we denote by $q$ the maximal projection in $N$ for which $\eta_\alpha$ converges uniformly on $q\UU(A)$ then it is a simple argument to show 
that $q$ exists and must commute with the normalizer of $A$.  Since $A$ is a regular subalgebra and $N$ is a factor we must then have that $q$ is 
a scalar.  However we have just shown that $0 \not= p \leq q$ and hence we must have $q = 1$, i.e., $\eta_\alpha$ converges uniformly on 
all of $\UU(A)$.

\end{proof}

Note that if $B$ above is taken completely atomic (e.g. if $B = \C$) then no corner of a diffuse subalgebra can embed into $B$ inside $N$ and hence we 
conclude that no sequence $\{ u_n \}_n$ can satisfy the hypotheses of the above theorem.  

We will now show that if no such sequence exists and 
if $\Lambda$ acts on $A$ such that $N \cong A \rtimes \Lambda$ then $\Lambda$ must have the Haagerup property.  We do this by producing a 
sequence of positive definite functions which satisfies the hypotheses of Theorem \ref{thm:notuniformhaagerup}. 

\begin{thm}\label{thm:nocornerhaagerup}
Suppose $\G$ is a countable discrete group in the class $\CC$, and $\sigma: \G \rightarrow {\rm Aut}(B)$ is a trace preserving action on a 
finite von Neumann algebra such that $B \rtimes \G$ is a $II_1$ factor and the inclusion $( B \subset B \rtimes \G )$ has the Haagerup property from below.  
Suppose that $\Lambda$ is a countable discrete group which acts freely by measure preserving transformations on a standard 
probability space $(Y, \nu)$, let $A = L^\infty(Y, \nu)$.  If $\theta: A \rtimes \Lambda \rightarrow N$ is an isomorphism such that $\theta(A) 
\not\preceq_N B$ then $\Lambda$ has the Haagerup property.
\end{thm}

\begin{proof}
We will use the same notation as at the beginning of Theorem \ref{thm:tech}.  Let $\alpha_k \rightarrow \infty$ and define a sequence of 
positive definite functions $\varphi_k: \Lambda \rightarrow [0, 1]$, by $\varphi_{2k}(\lambda) = \tau( \eta_{\alpha_k}^j(\theta(u_\lambda) ) 
\theta(u_\lambda^*) )$ and $\varphi_{2k + 1}(\lambda) = \tau ( \Psi_k( \theta( u_\lambda) ) \theta( u_\lambda^* ))$.  Note that by Lemma 
\ref{lem:haagerupsemigroup} we may assume that $\varphi_{2k}(\lambda)$ is an increasing function of $k$, for all $\lambda \in \Lambda$.  Also 
note that $\varphi_k$ converges to $1$ uniformly on a set $S \subset \Lambda$ if and only if $\eta_\alpha$ and $\Psi_k$ converge to the 
identity uniformly on $\{ \theta(u_\lambda) | \lambda \in S \}$.

Let $\lambda_n \in \Lambda$ be a sequence of group elements and set $u_n = \theta(u_{\lambda_n}) \in N$. If $\lim_{k \rightarrow \infty} \sup_n 
\| u_n - \Psi_k ( u_n ) \|_2 = 0$ and $\lim_{\alpha \rightarrow \infty} \sup_n \| u_n - \eta_\alpha (u_n) \|_2 = 0$ then it follows from 
Theorem \ref{thm:tech} that either there exists some $a \in \UU(A)$ such that $\{ u_n a u_n^* \}_n$ is not $\| \cdot \|_2$-precompact or else 
$\{ u_n \}_n$ is $\| \cdot \|_2$-precompact.

The first case cannot occur since then we would have that $\lim_{k \rightarrow \infty} \sup_n \| u_n a u_n^* - \Psi_k ( u_n a u_n^* ) \|_2 = 
0$, $\lim_{\alpha \rightarrow \infty} \sup_n \| u_n a u_n^* - \eta_\alpha (u_n a u_n^*) \|_2 = 0$, $\{ u_n a u_n^* \}$ is not $\| \cdot 
\|_2$-precompact, and $u_n a u_n^* a_0 u_n a^* u_n^* = a_0$, for all $n \in \N$.  Hence Theorem \ref{thm:tech} would apply and we would obtain 
a contradiction.  Thus we must have that $\{ u_n \}_n$ is $\| \cdot \|_2$-precompact, i.e., $\{ \lambda_n \}_n$ is a finite set.

Since $\{ \lambda_n \}_n$ must be finite it follows that the sequence $\varphi_k$ satisfies the hypotheses of Theorem 
\ref{thm:notuniformhaagerup} and hence $\Lambda$ has the Haagerup property.

\end{proof}

Before we proceed with applications of the above theorems we will show that the conclusion above can be strengthened to show that not just 
$\Lambda$ but actually all of $N$ (and hence also $\G$) has the Haagerup property.

\section{The Haagerup property for group-measure space constructions}

The following lemma is adapted from Corollary 4.2 in \cite{ioanaT} using an argument in \cite{petersonpopa}.  We are grateful to Adrian Ioana 
for suggesting this approach.

\begin{lem}\label{lem:reassemble}
Let $\G$ be a countable discrete group, and $\sigma: \G \rightarrow \textrm{Aut} (A)$ a free ergodic action.  Let $N = A \rtimes \G$ be the 
resulting group-measure space construction.

Suppose $c_0 < 1$, and there exist unital, tracial c.p. maps $\Phi_n: N \rightarrow N$ such that $\Phi_n$ is positive as an operator on $L^2N$ 
for all $n \in \N$, and $\Phi_n$ satisfy the conditions:
\vskip .1in
\noindent
$(a)$.  $\lim_{n \rightarrow \infty} \| x - \Phi_n(x) \|_2 = 0$, for all $x \in N$
\vskip .1in
\noindent
$(b)$.  $\tau( \Phi_n(y)  y^*) \leq c_0 \| y \|_2^2$, for all $y \in (A)_1$ off of a $\| \cdot \|_2$-compact set, i.e., for all $\kappa > 0$, 
there exists $K \subset A$ finite, such that if $y \in (A)_1$ and $\| P_{ \overline{\rm sp} K } (y) \|_2^2 < \kappa$, then $\tau(\Phi_n(y) y^*) 
< c_0 \| y \|_2^2 + \kappa$.
\vskip .1in
  Then the inclusion $(A \subset N)$ has the Haagerup property from below, i.e., there exists a sequence unital c.p. map $\Phi_n':N \rightarrow N$ 
such that 
\vskip .1in
\noindent
$(a')$.  $\lim_{n \rightarrow \infty} \| \Phi_n'(x) - x \|_2 = 0$, for all $x \in N$.
\vskip .1in
\noindent
$(b')$.  $Au_\gamma$ is $\Phi_n'$ invariant for all $\gamma \in \G$, $n \in \N$.
\vskip .1in
\noindent
$(c')$.  For each $\gamma \in \G$, $\Phi_n'$ is a compact operator when restricted to $L^2Au_\gamma$.
\end{lem}

\begin{proof}
Given the sequence $\Phi_n: N \rightarrow N$ from above, by restricting this sequence to $A$ we have that $\lim_{n \rightarrow \infty} \| a - 
E_A \circ \Phi_n(a) \|_2$, for all $a \in A$ and $\lim_{n \rightarrow \infty} \sup_{a \in (A)_1} \| \sigma_\gamma(a) - E_A \circ \Phi_n( \sigma_\gamma(a) ) \|_2 
\leq \lim_{n \rightarrow \infty} 2 \| u_\gamma - \Phi_n(u_\gamma) \|_2^{1/2} = 0$.  

By \cite{ioanaT} we may augment the maps 
${E_A \circ \Phi_n}_{|A}$ to unital, tracial, c.p. maps $\Phi_n^0: N \rightarrow N$ which leave $Au_\gamma$ invariant for all $\gamma 
\in \G$ and which still satisfy conditions $(a)$ and $(b)$ above.  We will outline this process here, for full details see \cite{ioanaT}.  

Given $\Phi_n$ one 
associates a pointed $N$-$N$ Hilbert bimodule $(\HH, \xi)$ such that $\langle x \xi_n y, \xi \rangle = \tau(\Phi_n(x) y)$, for
each $x, y \in N$.  If we view $A = L^\infty(X, \mu)$ for some standard Borel probability space $(X, \mu)$ then by Lemma 2.1 in \cite{ioanaT}
there exists a probability measure $\nu$ 
on $X \times X$ such that 

($i$)  \ \ $\int{(f_1 \otimes f_2)d\nu} = \langle f_1 \xi f_2, \xi \rangle$ for each $f_1, f_2$ Borel functions on $X$.

($ii$)  \ \ $\| ( \g \times \g )_* \nu - \nu \| \leq 2 \| u_\g \xi - \xi u_\g \|$, for each $\g \in \G$.

($iii$)  \ \ $p_*^i\nu = \mu$, for $i = 1,2$ where $p^i$ is the projection $p^i(x_1, x_2) = x_i$.

If we then enumerate $\G$ by $\{ \g_i \}$ and take $0 < C_n < 1$, then we may set $\nu_0 = C_n\nu + \Sigma_{i = 1}^\infty (1 - C_n)2^{-i} (\g_i)_*\nu$
and notice that $\g_*\nu \sim \nu$, for each $\g \in \G$, and $\| \nu - \nu_0 \| \leq 2(1 - C_n)$.  Also note that we have
$p_*^i\nu_0 = \mu$, for $i = 1,2$.
Hence we may define $g_\g = (d\g_*\nu_0/d\nu_0)^{\frac{1}{2}} \in L^2(X \times X, \nu_0)_+$, for each $\g \in \G$.

If we view $L^2(X, \mu)$ as a Hilbert subspace of $L^2(X \times X, \nu_0)$ via the map $f \mapsto f \circ p^2$ and we denote 
by $E:L^2(X \times X, \nu_0) \rightarrow L^2(X, \mu)$ the orthogonal projection, then by Corollary 4.2 in \cite{ioanaT} the formula:
$$
\Phi_n^0(\Sigma_\g f_\g u_\g) = \Sigma_\g E((f \otimes 1)g_{\g^{-1}})u_\g
$$
defines a unital, tracial, c.p. map on $N$.

Moreover, if we take $C_n \rightarrow 1$, then since $\| \nu - \nu_0 \| \leq 2(1 - C_n) \rightarrow 0$ it follows that
both conditions $(a)$ and $(b)$ (for possibly a new $c_0 < 1$) will again be satisfied for $\Phi_n^0$.

Since $N$ is separable, by taking a subsequence we may assume that $\{ \Phi_n^0 \circ \Phi_{n-1}^0 \circ \cdots \circ \Phi_k^0(x) \}_n$ is a 
$\| \cdot \|_2$-Cauchy sequence for all $x \in N$, $k \in \N$, hence we may consider the sequence of unital, tracial, c.p. maps given by the 
infinite compositions $\Phi_k' = \cdots \circ \Phi_n^0 \circ \Phi_{n-1}^0 \circ \cdots \circ \Phi_k^0$.  Also by taking a subsequence we may 
assume that $\lim_{k \rightarrow \infty} \| x - \Phi_k'(x) \|_2 = 0$, for each $x \in N$, so that $\Phi_k'$ satisfies condition $(a')$ above.

As each $\Phi_n^0$ satisfies condition $(b')$ we have that this is the case also for the composition $\Phi_k'$.
The fact that $\Phi_k'$ also satisfies condition $(c')$ then follows since it is an infinite composition of maps which each shrinks the norm by 
a fixed amount off of a compact set.  See for example the Proof of $(1.2) \implies (0.2)$ in \cite{petersonpopa}. 
\end{proof}

As a consequence, we see that Definition \ref{defn:relhaagerup} is a reasonable one in that a group-measure space construction has the 
Haagerup property if and only if the Cartan inclusion has the Haagerup property from below and has the Haagerup property relative to the 
Cartan subalgebra.

\begin{cor}\label{cor:vnhaagerupprop}
Let $\G$ be a countable discrete group, and $\sigma: \G \rightarrow \textrm{Aut} (A)$ a free ergodic action on a standard probability space.  
Let $N = A \rtimes \G$ be the
resulting group-measure space construction.  Then $N$ has the Haagerup property if and only if $\G$ has the Haagerup property and the inclusion 
$(A \subset A \rtimes \G)$ has the Haagerup property from below.
\end{cor}

If the hypotheses of Theorem \ref{thm:nocornerhaagerup} are satisfied 
then Theorem \ref{thm:tech} together with Lemma \ref{lem:reassemble} shows that 
the inclusion $(A \subset N)$ has the Haagerup property from below.  On the other hand in Theorem \ref{thm:nocornerhaagerup} we have already 
shown that $\Lambda$ has the Haagerup property.  Thus we have shown that all of $N$ must have the Haagerup property.  We therefore obtain the 
following theorem.

\begin{thm}\label{thm:general}
Suppose $\G$ is a countable discrete group in the class $\CC$, and $\sigma: \G \rightarrow {\rm Aut}(B)$ is a trace preserving action on a 
finite von Neumann algebra such that $B \rtimes \G$ is a $II_1$ factor and the inclusion $( B \subset B \rtimes \G )$ has the Haagerup 
property from below. Suppose that $\Lambda$ is a countable discrete group which acts freely by measure preserving transformations on a standard 
probability space $(Y, \nu)$, let $A = L^\infty(Y, \nu)$.  If $\theta: A \rtimes \Lambda \rightarrow N$ is an isomorphism such that $\theta(A) 
\not\preceq_N B$ then $N$ (and hence also $\G$) has the Haagerup property.
\end{thm}

\section{Proofs of Theorems \ref{thm:introuniquegms}, \ref{thm:introwesuper}, and \ref{thm:intronogms}}

Using Theorem \ref{thm:general} and the intertwining techniques of Popa we are now in position to prove the theorems in the introduction.

\begin{proof}[Proof of Theorem \ref{thm:intronogms}]
Suppose $\Lambda$ is a countable discrete group which acts freely on a standard probability space $(Y, \nu)$ such that $L^\infty(Y, \nu) 
\rtimes \Lambda \cong L\G \overline \otimes M = N$.  Since $\G$ does not have the Haagerup property neither does $N$ and hence by Theorem 
\ref{thm:general} we must have that a corner of $L^\infty(Y, \nu)$ embeds into $M$ inside $N$.  In particular, since $L\G \subset M' \cap N$ we 
have that $L^\infty(Y, \nu)$ is not maximal abelian and thus we obtain a contradiction.
\end{proof}

\begin{proof}[Proof of Theorem \ref{thm:introuniquegms}]
Suppose $\Lambda$ is a countable discrete group which acts freely on a standard probability space $(Y, \nu)$ such that $L^\infty(Y, \nu) 
\rtimes \Lambda \cong L^\infty(X, \mu) \rtimes \G = N$.  Again, since $\G$ does not have the Haagerup property neither does $N$ and hence by 
Theorem \ref{thm:general} we must have that a corner of $L^\infty(Y, \nu)$ embeds into $L^\infty(X, \mu)$ inside $N$.  As these are both Cartan 
subalgebras inside a $II_1$ factor Theorem A.1 in \cite{popabetti} shows that they must be conjugate by a unitary in $N$.
\end{proof}

\begin{proof}[Proof of Theorem \ref{thm:introwesuper}]
Again suppose $\Lambda$ is a countable discrete group which acts freely on a standard probability space $(Y, \nu)$ such that $L^\infty(Y, \nu) 
\rtimes \Lambda \cong L^\infty(X, \mu) \rtimes \G = N$.  From Theorem \ref{thm:general} we must have that a corner of $L^\infty(Y, \nu) = A$ 
embeds into $L^\infty(X, \mu) \rtimes \G_2 = B$ inside $N$.  Therefore by Theorem A.1 in \cite{popabetti} (or see Theorem 2.1 in 
\cite{popamalleableI}) there exists non-zero projections $p \in A$ and $q \in B$, a normal $*$-homomorphism $\theta: pAp \rightarrow qBq$, and a 
non-zero partial isometry $v \in N$ such that
$$
\forall x \in pAp \ \ xv = v\theta(x)
$$
and $v^*v \in \theta(pAp)' \cap qNq$, $vv^* \in p(A' \cap N)p$.

Note that $\theta(pAp)$ is a Cartan subalgebra of $qBq$ by \cite{jonespopa}.

Since $B$ is a $II_1$ factor which is weakly amenable with constant $1$ and since $\G_2$ has an unbounded cocycle into a nonamenable 
representation, it follows from \cite{ozawapopaI} and \cite{ozawapopaII} that $B$ and hence $qBq$ also, have at most $1$ Cartan subalgebra up 
to unitary conjugation, i.e., there must exist $w \in \UU(qBq)$ such that $w\theta(pAp)w^* = qL^\infty(X, \mu)q$, in particular, since $N$ is a 
$II_1$ factor it follows that $A$ and $L^\infty(X, \mu)$ are unitarily conjugate in $N$.

Thus we have shown that the actions $\G \curvearrowright (X, \mu)$ and $\Lambda \curvearrowright (Y, \nu)$ are orbit equivalent and the result 
then follows from Theorem A in \cite{ioanaprofinite}.
\end{proof}

\end{document}